\documentclass{amsart}
\usepackage{amsmath}
\usepackage{amssymb}
\usepackage{amscd}
\usepackage{latexsym}
\usepackage{graphicx} 

\theoremstyle{plain}
\newtheorem{thm}{Theorem}

\newtheorem*{defin}{Definition}

\newtheorem*{claim}{Claim}

\newtheorem*{propo}{Proposition}

\theoremstyle{definition}

\newtheorem{remk}{Remark}

\def \Z {\mathbf{Z}}

\def \G {\Gamma}

\def \- {\!\smallsetminus\!}

\def \sw {\mathcal{SW}}

\def\fs{\mathfrak{s}}

\def\spinc{spin$^{\text{c}}$}

\begin{document}

\baselineskip.5cm
\title {A note on knot surgery}
\thanks{This work was partially supported by NSF RTG Grant DMS-0353717 and by DMS-0704091.}
\author[Nathan S. Sunukjian]{Nathan S. Sunukjian}
\address{Department of Mathematics, Stony Brook University \newline
\hspace*{.375in}Stony Brook, New York 11794}
\email{\rm{nsunukjian@math.sunysb.edu}}
%\date{March 4, 2009}
\begin{abstract} In this paper we clarify an issue in the knot surgery construction of Fintushel and Stern. Using knot surgery, they construct an infinite number of smooth structures on 4-manifolds satisfying certain conditions, but they do not explicitly work out the circumstances under which two manifolds that arise from their construction will fail to be diffeomorphic on the grounds of Seiberg-Witten theory. This paper fills in that gap.\end{abstract}
\maketitle

In their paper \cite{KL4M}, Fintushel and Stern produce infinite families of homeomorphic but non-diffeomorphic $4$-manifolds using a technique they call knot surgery, whereby the neighborhood of a torus is replaced with something homologically equivalent, but smoothly ``knotted''. Remarkably this process does not change the homeomorphism type of a 4-manifold, and equally remarkable is the effect on the Seiberg-Witten invariant. 

For the purposes of this paper, we shall think of the Seiberg-Witten invariant of $X$, denoted $SW_X$ as an element of the group ring $\Z[H_2(X)]$. This usage will be clarified below.

Specifically, Fintushel and Stern prove the following:

\begin{thm}[{\cite[Theorem 1.1]{KL4M}}] \label{FSknot} Let $T$ be an embedded torus in a simply connected $4$-manifold $X$ with $[T]^2=0$ and let $K$ be a knot in $S^3$. Suppose further that $\pi_1(X\setminus T) = 1$. Then $X$ is homeomorphic to the knot surgered manifold manifold $X_K := (X \setminus \nu T) \cup  (S_1 \times S^3 \setminus \nu K)$, the only requirement on the gluing being that the longitude of $K$ be taken to the meridian of $T$.

Moreover, $SW_{X_K}$ is obtained from $SW_X$ via multiplication by the symmetrized Alexander polynomial of $K$: 

\[\label{f:SWknot} SW_{X_K} = SW_X \cdot \Delta_K(2[T]) \]  
\end{thm}

%It is evident from this result that one can construct infinite families of exotic smooth manifolds by varying the knot $K$. What is not prima facie evident is that knots with two different Alexander polynomials will \emph{always} give non-equivalent knot surgeries. The purpose of this note is to clarify and resolve this issue.

%What is not prima facie evident from this result is that knots with two different Alexander polynomials will \emph{always} give non-equivalent knot surgeries. The key issue is making sense of \ref{f:SWknot}. The left side of the equation seem to lie in $\Z H_2(X_K)$, while the right hand side lies in $\Z H_2(X)$. What does it even mean for Seiberg-Witten theory to give a well defined invariant in $\Z H_2(X)$? The purpose of this note is to clarify and resolve this issue.

What is not prima facie evident from this result is that knots with two different Alexander polynomials will \emph{always} give non-equivalent knot surgeries. The purpose of this note is to clarify and resolve this issue.

\begin{thm}\label{main} If $K_1$ and $K_2$ are knots with different Alexander polynomials, then $X_{K_1}$ and $X_{K_2}$ cannot be diffeomorphic. 
\end{thm}

These subtleties arises because of the somewhat imprecise way we have described $SW_X \in \Z[H_2(X)]$ as an invariant of $X$. In this case, it should really be though of as an invariant \emph{up to automorphisms of} $\Z[H_2(X)]$. Here is why: The Seiberg-Witten invariant is typically defined as a map $\sw:$ \spinc $(X)  \rightarrow \Z$. We encode this information as an element of $\Z[H_2(X)]$ by defining $SW_X := \sum{\sw(\fs)PD(c_1(\fs))}$ where the sum is taken over all \spinc  structures on $X$. When we do knot surgery on $X$ to produce $X_K$, our new Seiberg-Witten invariant $SW_{X_K}$ is an element of $\Z[H_2(X_K)]$. We can think of this as an element in $\Z[H_2(X)]$ --- which is what we do implicitly in the knot surgery formula --- because $H_2(X)$ is isomorphic to $H_2(X_K)$. In fact, this isomorphism is canonical, \emph{but only with respect to the surgery}. Different knot surgeries, even surgeries that give diffeomorphic manifolds, will induce different isomorphisms of $H_2$, and hence we may manifest the resulting Seiberg-Witten invariants as different elements of $\Z[H_2(X)]$.

Consider the following illustrative example: 
Suppose $X$ is a 4-manifold containing two tori, $T_1$ and $T_2$, representing different homology classes such that there is a self-diffeomorphism of $X$ taking $T_1$ to $T_2$. For a fixed knot $K$, do knot surgery on $T_1$ and $T_2$ forming $X_1$ and $X_2$. Clearly knot surgery can be performed in such a way that these manifolds are diffeomorphic, but note that their Seiberg-Witten invariants, as elements in $\Z[H_2(X)]$, will be different. According to the knot surgery formula, if $c_1(\fs)$ is a basic class of $X$, then on $X_1$ we get new basic classes of the form $c_1(\fs) + n[T_1]$, whereas our new basic classes on $X_2$ are of the form $c_1(\fs) + n[T_2]$. However, the diffeomorphism of $X_1$ to $X_2$ induces an automorphism of $H_2(X)$ that takes $[T_1]$ to $[T_2]$ (and consequently takes $SW_{X_1}$ to $SW_{X_2}$).

In the case at hand, where $\Delta_{K_1} \neq \Delta_{K_2}$ and we want to show $X_{K_1}$ is not diffeomorphic to $X_{K_2}$, we will need to associate to each element of $\Z[H_2(X)]$ a quantity that is invariant under automorphisms of $H_2(X)$. If $\alpha = \sum a_i [h_i] \in \Z[H_2(X)]$, denote $\overline{\alpha} = \sum a_i [h_i]^{-1}$.  

\begin{defin}
%Suppose $H_2(X)$ is torsion free, $\alpha$ is an irreducible element of $\Z[H_2(X)]$, and $\phi$ is an automorphism of $\Z[H_2(X)]$. Define a map 
%	$\G_{\alpha,\phi} : \Z[H_2(X)] \to \Z$
%by      $x \mapsto  (\sharp$ of elements of $\left\{\phi^n(\alpha)|n\in\Z \right\}$ that can be factored out of $x$ counting multiplicity).
Suppose $H_2(X)$ is torsion free, $\alpha$ is an irreducible element of $\Z[H_2(X)]$, and $\phi$ is an automorphism of $\Z[H_2(X)]$, (which is the linear extension of an automorphism of $H_2(X)$). Given $x \in \Z[H_2(X)]$, define $\G_{\alpha,\phi}(x)$ as the number of elements of the form $\phi^n(\alpha)$ or $\phi^n(\overline{\alpha})$ for any $n\in \Z$ that can be factored out of $x$ counting multiplicity.

\end{defin}

This is a well defined invariant because $\Z[H_2(X)]$ is a UFD. Moreover, $\G$ has the following basic properties:
\begin{propo}  

\begin{itemize}
\item[(i)] For $a, b \in \Z[H_2(X)]$, we have  $\G_{\alpha,\phi}(ab) = \G_{\alpha,\phi}(a) + \G_{\alpha,\phi}(b)$
\item[(ii)] $\G_{\alpha,\phi} \circ \phi = \G_{\alpha,\phi}$
\item[(iii)] $\G_{\alpha,\phi}(\beta) = \G_{\alpha,id}(\beta)$ when $\alpha,\beta \in \Z[\langle[T]\rangle]$, and $\alpha$ has more than one term in its (unique) expansion as $\sum a_i [T]^i$.
\end{itemize}
\end{propo}

\begin{proof}

Only the third property deserves further comment. It is sufficient to show that if $\phi^n(\alpha)$ or $\phi^n(\overline{\alpha})$ can be factored out of $\beta$, then these factors are equal to $\alpha$ or $\overline{\alpha}$. 

Suppose, first of all, that $\phi^n(\alpha)$ can be factored out of $\beta$. Since $\beta$ can be factored into irreducibles that are in $\Z[\langle[T]\rangle]$, we have that $u\phi^n(\alpha)$ is in $\Z[\langle[T]\rangle]$ for some unit $u \in \Z[H_2(X)]$. Since $\alpha \in \Z[\langle[T]\rangle]$, we can write $u\phi^n(\alpha) = u\sum{a_i\phi^n([T]^i)}$ for $a_i \in \Z$.

Now, $u\sum{a_i\phi^n([T]^i)}$ being in $\Z[\langle[T]\rangle]$ implies that $u\phi^n([T]^i)$ is too (at least for the indices $i$ corresponding to non-zero terms in the sum), and since the summation must have more than one term by hypothesis, we can find $u\phi^n([T]^{i}) = \pm 1\cdot[T]^{j}$ and $u\phi^n([T]^{i'}) = \pm 1\cdot[T]^{j'}$ for some $i \neq i'$ and $j \neq j'$. (Recall that units of $\Z[H_2(X)]$ are all of the form $\pm 1 \cdot h$ for $h\in H_2(X)$.) Therefore, $\pm 1\cdot[T]^{-j'} = u^{-1}\phi([T]^{-i'})$, and this implies that $[T]^{j-j'} = \phi^n([T]^{i-i'})$. Since $\phi^n$ must preserve degree, we get that that $\phi^n([T]) = [T]$ or $[T]^{-1}$, and hence $\phi^n(\alpha) = \alpha$ or $\overline{\alpha}$. 

The same proof works when $\beta$ has a factor of the form $\phi^n(\overline{\alpha})$.

%Extend $[T]$ to a basis of $H_2(X)$, e.g. write $\Z[H_2(X)] = \Z[T,T^{-1},x_1,x_1^{-1},...,x_m,x_m^{-1}]$. Let $\psi : \Z[T,...,x_m^{-1}] \to \R[x_1,x_1^{-1},...,x_m,x_m^{-1}]$ be a homomorphism that takes $[T] \mapsto [0]$ and is the natural map on all other elements. Since $\psi(\Z[\langle[T]\rangle]) = \Z$, we have that $\psi(u\phi^n(\alpha)) = \psi(u)\psi(\phi^n(\alpha)) \in \Z$. Since $\psi(u)$ is a unit, this fact implies that $\psi(\phi^n(\alpha))$ is a unit in $\R[x_1,...,x_m^{-1}]$. Units in this group ring are of the form $ax_1^{n_1}...x_m^{n_m}$ where $a \in \R$. So $\phi^n(\alpha)$  must either be an element of $\Z[\langle[T]\rangle]$ --- which cannot be since $\phi^n(\alpha) \neq \alpha$ --- or an element of the form $iT^{n_t}x_1^{n_1}...x_m^{n_m}$ where $i,n_* \in \Z$. But this cannot happen either because if $i=1$, then $\phi^n(\alpha)$ is not an irreducible factor (it is a unit), and if $i \neq 1$, then the coefficients of $\Delta_k$ do not sum to 1, which cannot happen for an Alexander polynomial. 
\end{proof}

These properties are sufficient to prove theorem 2.

\begin{proof}[Proof of Theorem \ref{main}]
To simplify our notation, write $\Delta_K$ for $\Delta_K(2[T])$. Assume $\Delta_{K_1} \neq \Delta_{K_2}$ but that $X_{K_1}$ is diffeomorphic to $X_{K_2}$. We will derive a contradiction.
According to the knot surgery formula, the Seiberg-Witten invariants of $X_{K_1}$ and $X_{K_2}$ are $SW_X\,\cdot \,\Delta_{K_1} $ and $SW_X\,\cdot \, \Delta_{K2}$ respectively. A diffeomorphism $\phi : X_{K_1} \rightarrow X_{K_2}$ induces an automorphism $\phi_* : \Z[H_2(X_{K_1})] \rightarrow \Z[H_2(X_{K_2})]$ where $\phi_*(SW_X\,\cdot \,\Delta_{K_1}) = SW_X\,\cdot \,\Delta_{K_2}$

\begin{claim}
If $\Delta_{K_1} \neq \Delta_{K_2}$, then (without loss of generality) there is an irreducible factor $\alpha$ of $\Delta_{K_1}$ such that $\Gamma_{\alpha,id}(\Delta_{K_1}) > \Gamma_{\alpha,id}(\Delta_{K_2})$

\end{claim}
\begin{proof}
The only way this can fail to be be true is if there exists a factor $\alpha$ of $\Delta_{K_1}$ that is not present in $\Delta_{K_2}$ with the same multiplicity, but the sums of the multiplicities of $\alpha$ and $\overline{\alpha}$ in $\Delta_{K_1}$ and $\Delta_{K_2}$ are the same. This cannot happen because Alexander polynomials are symmetric: there exists a factorization such that the multiplicity of any factor $\alpha$ is equal to the multiplicity of $\overline{\alpha}$ (this is trivially true if $\alpha = \overline{\alpha}$).  
\end{proof}

%Since $\Delta_{K_1} \neq \Delta_{K_2}$, we can choose $\alpha$ to be an irreducible element of $\mathbb{Z}[\langle[T]\rangle]$ that divides $\Delta_{K_1}$ with a greater multiplicity than it divides $\Delta_{K_2}$. In other words $\Gamma_{\alpha,id}(\Delta_{K_1}(2[T])) > \Gamma_{\alpha,id}(\Delta_{K_2})$. 

For this choice of $\alpha$, via property (iii), this claim implies that $\Gamma_{\alpha,\phi_*}(\Delta_{K_1}) > \Gamma_{\alpha,\phi_*}(\Delta_{K_2})$.

To the equality $\phi_*(SW_X\,\cdot \,\Delta_{K_1}) = SW_X\,\cdot \,\Delta_{K_2}$ we apply $\G_{\alpha,\phi_*}$ (here shorthanded as $\G$) and use properties (i) and (ii) above:
\begin{align*}
\G(\phi_*(SW_X\cdot \Delta_{K_1})) &= \Gamma(SW_X\cdot \Delta_{K_2})\\
\G(\phi_*(SW_X)) + \G(\phi_*(\Delta_{K_1}))   &=   \G(SW_X) + \Gamma(\Delta_{K_2})\\
 \G(\Delta_{K_1}) &= \G(\Delta_{K_2})
\end{align*}

This, however, contradicts our choice of $\alpha$.

\end{proof}

\begin{remk} An essential hypothesis of this theorem was that $H_2(X)$ be torsion free: Otherwise $\Z[H_2(X)]$ is not a UFD and we cannot define $\G$. Note, however, that in the case $H_2(X)$ has torsion, the same proof can be carried out as long as the image of $SW_X$ in $\Z[H_2(X)/tor]$ is non-trivial. Simply replace every instance of $\Z[H_2(X)]$ above with $\Z[H_2(X)/tor]$.
\end{remk}

\begin{remk} Fintushel and Stern have applied the knot surgery technique on \textit{null-homologous} tori to produce exotic embeddings of surfaces in a 4-manifold, \cite{surfaces}\cite{addendum}. This process is called rim-surgery. The above proof can also be applied to rim surgery to show that any two knots with different Alexander polynomials will give rise to inequivalent rim-surgeries. The details for the relevant Seiberg-Witten invariant can be gleaned from \cite{FSS}.  
\end{remk}

\section*{Acknowledgments}
The author would like to thank Tim Cochran, Danny Ruberman, and Ron Fintushel (who was the author's disseration advisor when this note was written) for their comments and encouragement.


\begin{thebibliography}{MST}


\bibitem{KL4M} R. Fintushel \and R. Stern, {\em Knots, links, and $4$-manifolds}, Invent. Math. {\bf 134} (1998), 363--400.

\bibitem{surfaces} R. Fintushel \and R. Stern, {\em Surfaces in $4$-manifolds}, Math. Res. Letters {$\mathbf{4}$} (1997), 907--914.

\bibitem{addendum} R. Fintushel \and R. Stern, {\em  Surfaces in $4$-manifolds: Addendum}, math.GT/0511707.

\bibitem{FSS} R. Fintushel \and R. Stern \and N. Sunukjian, {\em Exotic group actions on simply connected smooth 4-manifolds}, J. Topology {\bf 2}  (2009), 769--822.


\end{thebibliography}
\end{document}